\documentclass[reqno,a4paper,11pt]{amsart}
\usepackage{amsthm,indentfirst}
\usepackage{kantlipsum}
\usepackage[all,poly]{xy}
\usepackage{amsfonts}
\usepackage[mathcal]{eucal}
\usepackage{amssymb}
\usepackage{amsmath}
\usepackage{mathrsfs}
\usepackage{color}
\usepackage[pagebackref,colorlinks]{hyperref}
\usepackage{enumerate}

\theoremstyle{plain}
\newtheorem {lemma}{Lemma}
\newtheorem {proposition}[lemma]{Proposition}
\newtheorem {theorem}[lemma]{Theorem}

\theoremstyle{definition}
\newtheorem{definition}[lemma]{Definition}

\newcommand{\Z}{\mathbb Z}
\newcommand{\C}{\mathbb C}
\newcommand{\M}{\operatorname{\mathbb M}}
\newcommand{\ann}{\operatorname{ann}}
\newcommand{\so}{\mathbf{s}}
\newcommand{\ra}{\mathbf{r}}
\newcommand{\ol}{\overline}

\title{Baer and Baer *-ring characterizations of Leavitt path algebras}

\author{Roozbeh Hazrat}
\address{Centre for Research in Mathematics,
Western Sydney University,
Australia}\email{r.hazrat@westernsydney.edu.au}

\author{Lia Va\v s}
\address{Department of Mathematics, Physics and Statistics, University of the Sciences, Philadelphia, PA 19104, USA}
\email{l.vas@usciences.edu}

\thanks{We would like to thank Hannes Thiel for his insightful comments on $C^*$-algebras and the referee for a careful reading, encouraging comments and useful suggestions. The first author would like to acknowledge Australian Research Council grants DP150101598 and DP160101481. A part of this work was done at the University of M\"unster, where the first author was a Humboldt Fellow.}


\begin{document}



\begin{abstract}
We characterize Leavitt path algebras which are Rickart, Baer, and Baer $*$-rings in terms of the properties of the underlying graph. In order to treat non-unital Leavitt path algebras as well, we generalize these annihilator-related properties to locally unital rings and provide a more general characterizations of Leavitt path algebras which are locally Rickart, locally Baer, and locally Baer $*$-rings. Leavitt path algebras are also graded rings and we formulate the graded versions of these annihilator-related properties and characterize Leavitt path algebras having those properties as well. 

Our characterizations provide a quick way to generate a wide variety of examples of rings. For example, creating a Baer and not a Baer $*$-ring, a Rickart $*$-ring which is not Baer, or a Baer and not a Rickart $*$-ring, is straightforward using the graph-theoretic properties from our results. In addition, our characterizations showcase more properties which distinguish behavior of Leavitt path algebras from their $C^*$-algebra counterparts. For example, while a graph $C^*$-algebra is Baer (and a Baer $*$-ring) if and only if the underlying graph is finite and acyclic, a Leavitt path algebra is Baer if and only if the graph is finite and no cycle has an exit, and it is a Baer $*$-ring if and only if the graph is a finite disjoint union of graphs which are finite and acyclic or loops. 
\end{abstract}
 
\maketitle

\section{Introduction}

Leavitt path algebras associated to directed graphs were introduced in \cite{aap05, Ara_Moreno_Pardo} as the algebraic counterparts of the graph $C^*$-algebras and as generalizations of Leavitt algebras. The study of Leavitt path algebras grew rapidly and several directions of research emerged. One of these directions is the characterization of the ring-theoretic properties
of a Leavitt path algebra $L_K(E)$ in terms of the graph-theoretic properties of the graph $E$, i.e. results of the form 
\begin{center}
{\em $L_K(E)$ has (ring-theoretic) property $(P)$ if and only if  $E$ has (graph-theoretic) property $(P').$ }
\end{center}
While relevant in its own right, this line of research has also become a way to create rings with various predetermined properties. Namely, by choosing suitable graphs, one can produce prime rings which are not primitive (Kaplansky's Conjecture~\cite{abr}), simple rings which are not purely infinite simple (\cite{aap06}), and strongly graded rings which are not crossed-products (\cite{Hazrat_graded}), to mention just a few applications. The characterization theorems  
have been formulated and proven for a number of ring-theoretic properties: being simple, purely infinite simple, hereditary, exchange, semisimple, artinian, noetherian, directly finite, to name some of them, and, in particular, von Neumann regular (\cite{Abrams_Rangaswamy}), right semihereditary and right nonsingular (\cite{Ara_Moreno_Pardo, Ara_Goodearl, Mercedes}). Thus, the Leavitt path algebra characterization is known for most of the properties listed in the diagram 
from ~\cite[p. 262]{Lam} shown below. 
\begin{equation}\label{diagram}\tag{D}
\xymatrix{
&& (\, \text{Baer} \, ) \ar@2[d] \\
\Big( \, \txt{von Neumann\\regular} \, \Big   )
\ar@2[r] &   \Big(\,  \txt{right \\semihereditary}  \, \Big )
\ar@2[r] &  \Big(\,  \txt{right \\Rickart} \, \Big )
\ar@2[r] &  \Big(\,  \txt{right \\nonsingular} \, \Big )
} 
\end{equation}
However, the property of being Baer has not yet been characterized for Leavitt path algebras. In this paper, we fill this gap by producing a graph-theoretic property which characterizes when a Leavitt path algebra is Baer (Theorem \ref{Baer_characterization}).

Recall that a ring $A$ is Baer if the right (equivalently left) annihilator of any subset of $A$ is generated by an idempotent and that $A$ is right Rickart if the right annihilator of any element is generated by an idempotent. A left Rickart ring is defined analogously. These properties emerged from the consideration of  operator algebras by Kaplansky and Rickart. Kaplansky and Berberian considered these properties for general rings in order to find ``algebraic avenues'' (as Berberian puts it in \cite{Berberian}) into operator theory, i.e. to study algebraic counterparts of operator algebras using algebraic methods alone.   
The treatment of Leavitt path algebras as the algebraic counterparts of graph $C^*$-algebras, fits right into this trend. In this respect, the characterization of Baer Leavitt path algebras seems to be a particularly natural progression. 

Most algebras of operators as well as Leavitt path algebras are involutive. In the presence of an involution, the projections, self-adjoint idempotents, are ``vastly easier to work with than idempotents'' (\cite{Berberian}). If the word ``idempotent'' is replaced by the word ``projection'' in the definitions of Baer and Rickart rings, one obtains Baer $*$-rings and Rickart $*$-rings respectively. We also produce a graph-theoretic property which determines when a Leavitt path algebra is a Baer $*$-ring (Theorem \ref{Baer_star_characterization}). 

The properties of an involution of the field $K$ directly impact the involution-related properties of $L_K(E)$. 
As a consequence, when characterizing properties of Leavitt path algebras as involutive rings, the underlying field, not just the graph, also becomes relevant (see \cite[Theorem 3.3]{Gonzalo_Ranga_Lia}, for example). This contrasts characterization results for graph $C^*$-algebras for which the underlying field is fixed. Our Baer $*$-ring characterization (Theorem~\ref{Baer_star_characterization}) holds under the assumption that the involution on $K$ is positive definite and we show that this assumption is necessary. 

In many cases, a single property of the graph $E$ characterizes an algebraic property of both $L_K(E)$ and the graph $C^*$-algebra $C^*(E).$ 
There are some exceptions (for example in results characterizing being von Neumann regular or being prime) but generally such exceptions occur less often. Our results provide more examples of such exceptions. In particular, 
\begin{center}
\begin{tabular}{l}
{\em $L_K(E)$ is Baer if and only if  $E$ is finite and no cycle has an exit}\hskip.4cm while \\
{\em $C^*(E)$ is Baer if and only if  $E$ is finite and acyclic (has no cycles).}
\end{tabular}
\end{center}
Also, if $K$ is positive definite, 
\begin{center}
\begin{tabular}{ll}
{\em $L_K(E)$ is a Baer $*$-ring if and only if} & {\em $E$ is a finite disjoint union of}\\
& {\em graphs which are finite and acyclic or loops} while \\
{\em $C^*(E)$ is a Baer $*$-ring if and only if} & {\em  $C^*(E)$ is Baer\hskip.4cm  if and only if\hskip.4cm  $E$ is finite and acyclic.}
\end{tabular}
\end{center}
Also, a graph $C^*$-algebra is Rickart if and only if it is a Rickart $*$-ring and we show that there is a Rickart Leavitt path algebra which is not a Rickart $*$-ring. 

The graph-theoretic properties which appear in our main results (Proposition~\ref{Rickart_characterization} and Theorems \ref{Baer_characterization} and \ref{Baer_star_characterization}) make it straightforward to construct algebras with various properties such as the following: Rickart and not Baer, Baer and not Baer $*$, Rickart $*$ and not Baer nor regular, and Baer and not a Rickart $*$-ring. These examples illustrate the use of the characterization results more generally than just for Leavitt path algebras.   

A Rickart ring is necessarily unital, restricting the characterization of Rickart, Baer and Baer $*$-rings to Leavitt path algebras over graphs with finitely many vertices. However, while not always unital, Leavitt path algebras always have local units. So, one can ``localize'' ring-theoretic properties by considering those properties on corners generated by the local units as it has been done with noetherian and artinian properties in \cite{AAPM}, with unit-regular rings in \cite{Gonzalo_Ranga_Lia} and with directly finite rings in \cite{Lia_Traces}. This motivates our definitions: we say that a ring is {\em locally} Rickart if each corner is Rickart and we ``localize'' the other annihilator-related properties similarly. We formulate and prove our main results (Proposition~\ref{Rickart_characterization} and Theorems \ref{Baer_characterization} and \ref{Baer_star_characterization}) to provide characterizations of these more general properties without any restriction on the cardinality of the underlying graph.

Leavitt path algebras are also naturally graded by the group of integers. Many ring-theoretic properties have been adapted to graded rings (for example graded regular in \cite{Hazrat_regular} or graded directly finite in \cite{Roozbeh_Lia}). We adapt the properties from diagram ~(\ref{diagram}) to graded rings and show that the same relations from diagram (\ref{diagram}) continue to hold (Proposition \ref{graded_properties}). Our main results (Proposition~\ref{Rickart_characterization} and Theorems \ref{Baer_characterization} and \ref{Baer_star_characterization}) characterize the graded versions of the annihilator-related properties of Leavitt path algebras as well. It is interesting to point out the following. 
\begin{center}
\begin{tabular}{l}
{\em $L_K(E)$ is graded Baer if and only if $L_K(E)$ is Baer}\hskip.9cm while \\
{\em $L_K(E)$ can be a graded Baer $*$-ring without being a Baer $*$-ring.}
\end{tabular}
\end{center}
If the word ``Baer'' is replaced with the word ``Rickart'' in the sentence above, the claim continues to hold. In addition, while Leavitt path algebras satisfy Handelman's Conjecture (stating that a $*$-regular ring is necessarily directly finite and unit-regular) as shown in \cite{Gonzalo_Ranga_Lia}, we show that the graded version of Handelman's Conjecture fails for Leavitt path algebras by producing an algebra which is graded $*$-regular and neither graded unit-regular nor graded directly finite. 
 
The paper is organized as follows. In \S\ref{section_graded_and_involutive}, we formulate graded versions of the annihilator-related properties, prove some basic properties of those (including the graded version of diagram (\ref{diagram})) and generalize the properties of the annihilators from unital to locally unital rings. In \S\ref{section_positive_definite}, we focus on some properties of the involution of a Leavitt path algebra and provide an alternative proof of \cite[Proposition 2.4]{Gonzalo_Ranga_Lia} stating that a Leavitt path algebra is positive definite if and only if the underlying field is positive definite (Proposition \ref{positive_definite}). In \S\ref{section_baer}, we characterize Rickart, Baer, Baer $*$, locally Rickart, locally Baer, locally Baer $*$, and graded Rickart, graded Rickart $*$, graded Baer and graded Baer $*$ Leavitt path algebras (Proposition~\ref{Rickart_characterization} and Theorems \ref{Baer_characterization} and \ref{Baer_star_characterization}) and present some applications of our results and various examples in Remarks \ref{seven_remarks}. We conclude the paper by considering Leavitt path algebras which are Rickart $*$-rings and posing a question in \S\ref{section_questions}.

\section{Annihilator-related conditions for graded and locally unital rings}\label{section_graded_and_involutive}

\subsection{Rickart and Baer rings} For a subset $X$ of a ring $A$ (not necessarily unital), the right annihilator $\ann_r(X)$ of $X$ in $A$ denotes the set of elements $a\in A$ such that $xa=0$ for all $x\in X.$ The left annihilator $\ann_l(X)$ is defined analogously. It is straightforward to check that $\ann_r(X)$ is a right and $\ann_l(X)$ is a left ideal of $A$. If $B$ is a subring of $A$ and $X\subseteq B$, we use $\ann^B_r(X)$ to denote the set $\ann_r(X)\cap B$ of elements that annihilate $X$ from the right in $B.$ The analogous notation $\ann^B_l(X)$ is used for the left annihilator of $X$ in $B$. 

Recall that $A$ is {\em right Rickart} if $\ann_r(x)$ is generated by an idempotent as a right ideal for any $x\in A,$ {\em left Rickart} if the analogous condition holds for the left annihilators of elements, and {\em Rickart} if it is both left and right Rickart. If a ring is right Rickart, the idempotent which generates the right annihilator of zero is a left identity. Consequently, Rickart rings are necessarily unital. 

A ring $A$ is said to be an involutive ring or a $*$-ring, if it has an involution $*,$ an anti-automorphism of order two. A $*$-ring $A$ is left Rickart if and only if $A$ is right Rickart since  $(\ann_r(x))^*=\ann_l(x^*)$ for any $x\in A.$ Thus, a left or right Rickart ring with involution is unital.  

A unital ring $A$ is {\em Baer} if $\ann_r(X)$ (equivalently $\ann_l(X)$) is generated by an idempotent for any $X\subseteq A.$ This condition is left-right symmetric since $\ann_l(\ann_r(\ann_l(X)))=\ann_l(X).$ 
 
\subsection{Rickart * and Baer *-rings} If $A$ is a $*$-ring, the projections (self-adjoint idempotents) take over the role of idempotents. A $*$-ring $A$ is said to be a {\em Rickart $*$-ring} if $\ann_r(x)$ is generated by a projection for any $x\in A.$ This condition is left-right symmetric (since $(\ann_r(x))^*=\ann_l(x^*)$) and a Rickart $*$-ring is necessarily unital. The projection which generates the annihilator of an element is necessarily unique since $pA=qA$ for $p,q$ projections implies that $p=q.$ Note that idempotents do not necessarily have this property.  

A $*$-ring $A$ is a {\em Baer $*$-ring} if $\ann_r(X)$ (equivalently $\ann_l(X)$) is generated by a projection for any $X\subseteq A.$ Berberian's book \cite{Berberian} contains a detailed and comprehensive treatment of Baer $*$-rings. In \cite{Berberian_web}, one can find more details on Rickart, Baer and Rickart $*$-rings as well.

If $A$ is a $*$-ring, the matrix ring  $\M_n(A)$ also becomes an involutive ring with the $*$-transpose involution given by $(a_{ij})^*=(a_{ji}^*).$ 
If a $*$-ring $A$ is also a $K$-algebra for some commutative, unital $*$-ring $K,$ then $A$ is a {\em
$*$-algebra} if $(kx)^*=k^*x^*$ for $k\in K$ and $x\in A.$

An involution $*$ on $A$ is \emph{positive definite} if, for any $n$ and any $x_1, \dots,
x_n\in A$, $\sum_{i=1}^n x_ix_i^* = 0$ implies $x_i=0$ for each $i=1,\ldots,
n.$ If this condition holds for positive integers less than or equal to $n,$ the involution is {\em $n$-proper}. A $1$-proper involution is simply said to be proper.  
A $*$-ring with a positive definite ($n$-proper) involution is said to be a {\em positive definite ($n$-proper)} ring.
 
\subsection{Graded rings}
If $\Gamma$ is an abelian group, a ring $A$ is a \emph{$\Gamma$-graded ring} if $ A=\bigoplus_{ \gamma \in \Gamma} A_{\gamma}$ such that each $A_{\gamma}$ is
an additive subgroup of $A$ and $A_{\gamma}  A_{\delta} \subseteq
A_{\gamma + \delta}$ for all $\gamma, \delta \in \Gamma$. 
The elements of $A^h=\bigcup_{\gamma \in \Gamma} A_{\gamma}$ are the \emph{homogeneous elements} of $A.$ If $A$ is an algebra over a field $K$, then $A$ is  a \emph{graded algebra} if $A$ is a graded ring and 
$A_{\gamma}$ is a $K$-vector subspace for any $\gamma \in \Gamma$. 

A $\Gamma$-graded ring $A$ is \emph{trivially graded} if $A_0=A$ and $A_\gamma=0$ for $0\neq\gamma \in \Gamma$. Note that any ring can be trivially graded by any abelian group.

For a $\Gamma$-graded ring $A$ and $(\gamma_1,\dots,\gamma_n)$ in $\Gamma^n$, $\M_n(A)(\gamma_1,\dots,\gamma_n)$ denotes the $\Gamma$-graded ring $\M_n(A)$ with the $\delta$-component consisting of the matrices $(a_{ij})\in\M_n(A)$ such that $a_{ij}\in A_{\delta+\gamma_j-\gamma_i}$ for $i,j=1,\ldots, n$ (more details in~\cite[Section 1.3]{Hazrat_graded}). 

A \emph{graded right $A$-module} is a right $A$-module $M$ with a direct sum decomposition $M =\bigoplus_{\gamma\in\Gamma} M_\gamma$ where $M_\gamma$
is an additive subgroup of $M$ such that $M_\gamma A_\delta \subseteq M_{\gamma+\delta}$ for all $\gamma,\delta\in \Gamma$. In this case, for $\delta\in\Gamma,$ the
$\delta$-\emph{shifted} graded right $A$-module $M(\delta)$ is defined as $M(\delta) =\bigoplus_{\gamma \in \Gamma} M(\delta)_\gamma,$  where $M(\delta)_\gamma = M_{\delta+\gamma}.$ 
A right $A$-module homomorphism $f$ of graded right $A$-modules $M$ and $N$ is a {\em graded homomorphism} if $f(M_\gamma)\subseteq N_\gamma$
for any $\gamma\in \Gamma$. 

A \emph{graded right ideal} of $A$ is a right ideal $I$ such that $I =\bigoplus_{\gamma\in\Gamma} I\cap A_\gamma.$ A right ideal $I$ of $A$ is a graded right ideal if and only if $I$ is generated by homogeneous elements. This property implies that $\ann_r(X)$ is a graded right ideal of $A$ for any set $X$ of homogeneous elements of $A$. Graded left ideals and graded ideals are defined similarly. 

If $A$ is a $\Gamma$-graded ring, a graded free right $A$-module is defined as a graded right module which is a free right $A$-module with a
homogeneous basis (see \cite[Section 1.2.4]{Hazrat_graded}). A graded free left module is defined analogously.  
If $A$ is a $\Gamma$-graded ring and $\gamma_1, \ldots,\gamma_n\in\Gamma$, then $A(\gamma_1)\oplus \dots \oplus A(\gamma_n)$  
is a graded free right $A$-module. Conversely, any finitely generated graded free right $A$-module is of this form. 
A graded projective right $A$-module is a graded module which is graded isomorphic to a direct summand of a graded free right $A$-module (see \cite[Proposition 1.2.15]{Hazrat_graded} for an equivalent definition).  

\subsection{Graded Rickart and Baer rings}
We adapt the definitions of Rickart and Baer rings to graded rings and show some properties of graded Rickart and Baer rings.   

\begin{definition}\label{definition_graded_Rickart_and_Baer}
A $\Gamma$-graded ring $A$ is a \emph{graded right Rickart} ring 
if for any homogeneous element $x$, $\ann_r(x)$ is generated by a homogeneous idempotent. A graded left Rickart ring and a graded Rickart ring are defined analogously. A unital, graded ring $A$ is a \emph{graded Baer} ring if for any set $X$ of homogeneous elements, $\ann_r(X)$ (equivalently $\ann_l(X)$) is generated by a homogeneous idempotent.  
\end{definition}

Analogously to the non-graded case, one can show that the definition of a graded Baer ring is left-right symmetric  (see \cite[Proposition 7.46]{Lam}) and that a graded Rickart ring is unital (with an idempotent that generates $\ann_r(0)$, or equivalently $\ann_l(0),$ as the identity). Also note that Definition \ref{definition_graded_Rickart_and_Baer} reduces to the usual definitions if the graded ring is trivially graded.  

The next lemma shows that unital, graded right Rickart rings can be characterized by the properties which are analogous to those in the non-graded case.  
\begin{lemma}\label{graded_Rickart_lemma}
Let $A$ be a unital, $\Gamma$-graded ring. The following conditions are equivalent.
\begin{enumerate}
\item $A$ is a graded right Rickart ring.
 
\item Every graded principal right ideal of $A$ is projective.

\item Every graded principal right ideal of $A$ is graded projective.
\end{enumerate}
As a consequence, if $A$ is a graded ring which is right Rickart, then it is graded right Rickart. 
\end{lemma}
\begin{proof}
(1) $\Rightarrow$ (2). Let $x\in A_\gamma$ for some $\gamma\in \Gamma$. Then $xA$ is a graded right ideal of $A$ with the grading given by 
$(xA)_\delta=xA_{\delta-\gamma}$ and the multiplication $L_x$ by $x$ on the left, is a graded homomorphism $A\to A$ with the image $xA(\gamma).$ Since $A$ is graded right Rickart, $\ann_r(x)=eA$, for some idempotent $e\in A_0$. So, the left multiplication $L_e: A\rightarrow eA$ given by $a\mapsto ea$ is a graded map 
and the short exact sequence
\[0\longrightarrow \ann_r(x) \stackrel{i}{\hookrightarrow}  A \stackrel{L_x}\longrightarrow xA(\gamma) \longrightarrow 0\] of graded right $A$-modules splits as
$L_ei$ is the identity map on $\ann_r(x).$ Thus, $xA(\gamma)$ is graded projective. A shift of a graded projective module is also graded projective, so $xA$ is graded projective and, hence, projective. 

(2) $\Rightarrow$ (3). If a graded right ideal is projective then it is graded projective by \cite[Proposition~1.2.15]{Hazrat_graded}. 

(3) $\Rightarrow$ (1). If $x\in A^h$ and $xA$ is graded projective, then $xA(\gamma)$ is graded projective also. Thus, there is a graded homomorphism $\phi: A\to \ann_r(x)$ such that $\phi(y)=y$ for all $y\in \ann_r(x).$ Since $A$ is unital, $\phi(1)=e$ is an idempotent such that $\ann_r(x)=eA$. The idempotent $e$ is homogeneous as $e=\phi(1)\in \phi(A_0)\subseteq A_0.$

The last sentence of the lemma holds since, if $A$ is right Rickart, then any principal right ideal is projective and so condition (2) holds.  
\end{proof}

A graded ring $A$ is graded Baer if and only if it is graded Rickart and the lattice of  principal right ideals generated by homogeneous idempotents is complete. This statement and its proof are completely analogous to the statement and the proof of \cite[Proposition 1.21]{Berberian_web}. Using the graded-version of this statement, one can show that a  graded ring which is Baer is graded Baer. The converse does not hold as the following example illustrates. The group ring $\Z[\Gamma]$ is not Baer for any finite group $\Gamma$ (\cite[Theorem 2.4]{Yi_Zhou}). If $\Gamma$ is any finite abelian group, it is direct to show that $\Z[\Gamma]$ is graded Baer for the grading with $\Z[\Gamma]_\gamma=\{m\gamma\mid m\in \Z\}$ for $\gamma\in \Gamma$. 
 
We point out another property of graded right (left) Rickart and Baer rings. If $A$ is graded right Rickart, then $A_0$ is right Rickart. Indeed, if $x\in A_0$ and $\ann_r(x)=eA$ for an idempotent $e\in A_0,$ we claim that $\ann_r^{A_0}(x)=eA_0.$ If $xr=0$ for $r\in A_0,$ then $r=es$ for some $s\in A.$ But since $r,e\in A_0$, all the homogeneous components of $s$ from $A_\gamma$ for $\gamma\neq 0$ have to be trivial. Hence $s\in A_0$ and so $r\in eA_0.$ The converse is straightforward. Similarly, if  $A$ is graded left Rickart, then $A_0$ is left Rickart and if  $A$ is graded Baer, then $A_0$ is Baer.

We now establish the implications of diagram (\ref{diagram}) for graded rings. Let us introduce the graded versions of the remaining properties of the diagram. 

A graded ring is \emph{graded right (semi)hereditary} if any (finitely generated) graded right ideal is graded projective. 
A graded ring is \emph{graded von Neumann regular} if for any homogeneous element $x$, there is a homogeneous element $y$ such that $xyx=x$.  
A graded ring is \emph{graded right nonsingular} if for any nonzero, homogeneous element $x$, there is a nontrivial, graded right ideal $I$ such that $\ann_r(x) \cap I =0$. 

A unital, graded right semihereditary ring is graded right Rickart by Lemma \ref{graded_Rickart_lemma}. A graded Baer ring is graded Rickart by definition. We prove the remaining implications of diagram (\ref{diagram}) for unital, graded rings. 

\begin{proposition}\label{graded_properties}
Let $A$ be a unital, $\Gamma$-graded ring. 
\begin{enumerate}[\upshape(1)]

\item If $A$ is graded von Neumann regular then $A$ is graded semihereditary.

\item The ring $A$ is graded right semihereditary if and only if 
$\M_n(A)(\gamma_1,\dots,\gamma_n)$ is graded right Rickart, for any $n=1,2,\ldots$ and any $\gamma_1,\dots,\gamma_n \in \Gamma$.  

\item If $A$ is graded right Rickart, then $A$ is right nonsingular. 
\end{enumerate}
\end{proposition}
The proof follows the proofs of analogous statements for the non-graded case, \cite[Theorem~1.1]{Goodearl_book} and \cite[Proposition~5.2]{Handelman} (or \cite[Proposition 7.63]{Lam}). We include the proof for completeness.
\begin{proof}
(1) If $A$ is graded regular, the right ideal generated by a homogeneous element $x\in A$ is generated by a homogeneous idempotent. Indeed, if $y\in A^h$ is such that $xyx=x,$ then $xy$ is a homogeneous idempotent and $xA=xyA$. Therefore, all principal graded right ideals are projective. We show that the same holds for graded right ideals generated by two elements. Consider $xA+yA$ for $x,y\in A^h.$ Let $e=e^2\in A_0$ be such that $xA=eA.$ Then $y-ey \in A^h\cap (xA+yA)$ and so $xA+yA=eA+(y-ey)A$. Let $f=f^2\in A_0$ be such that $(y-ey)A=fA$. This implies that $ef=0$, and so $g=f-fe$ is a homogeneous idempotent with $ge=eg=0$ and $eA+gA=(e+g)A$. Since $fg=g$ and $gf=f$, it follows that $gA=fA=(y-ey)A.$ Consequently, $xA+yA=eA+gA=(e+g)A$. The claim for any finitely generated graded right ideal follows by induction. The claim for graded left ideals is showed analogously.   

(2) Begin semihereditary is a Morita invariant property (\cite[Corollary 18.6]{Lam}). Using an analogous proof, one can show that being graded semihereditary is a graded Morita invariant property. Thus, if $A$ is graded right semihereditary, then $\M_n(A)(\gamma_1,\dots,\gamma_n)$ is graded right semihereditary and thus graded right Rickart. 

For the converse, suppose that $I$ is a graded right ideal generated by $x_1,\ldots, x_n$ with $x_i\in A_{\gamma_i}$, for $i=1,\ldots, n$. Let $R=\M_n(A)(\gamma_1,\dots,\gamma_n)$ and consider $x=x_1e_{11}+x_2e_{12}+\dots+x_ne_{1n}$ where $e_{ij}$ are the standard matrix units. By definition, $x$ is in $R_{\gamma_1}$ (\cite[\S 1.3.1]{Hazrat_graded} has more details). The graded right ideal $xR=Ie_{11}+Ie_{12}+\dots+Ie_{1n}$ is a graded projective right $R$-module by Lemma~\ref{graded_Rickart_lemma}. On the other hand, since $R$ is a graded free right $A$-module, $xR$ is a graded projective right $A$-module. Since $xR$ is graded isomorphic to $I^n,$ $I$ is a graded summand of a free right $A$-module and so $I$ is a graded projective right $A$-module.

(3) Let $0\neq x\in A^h$. Then $\ann_r(x)=eA$ for an idempotent $e\in A_0$ and the graded right ideal $(1-e)A$ is such that $\ann_r(x) \cap (1-e)A=0$. Since $x\neq 0$ and $A$ is unital, $eA\subsetneq A$ and so $(1-e)A\neq 0.$
\end{proof}

\subsection{Graded Rickart *-ring and graded Baer *-rings}\label{subsection_graded_star}
In \cite{Roozbeh_Lia}, a $\Gamma$-graded ring $A$ with involution is said to be a \emph{graded $*$-ring} if $A_\gamma ^*\subseteq A_{-\gamma}$ for every $\gamma\in \Gamma.$ In this case, considering projections in place of idempotents leads to the following definitions.  

\begin{definition}\label{definition_graded_Rickart_and_Baer_star}
A graded $*$-ring $A$ is a \emph{graded Rickart $*$-ring} if the right (equivalently left) annihilator of any homogeneous element is generated by a homogeneous projection. 
A graded $*$-ring $A$ is a \emph{graded Baer $*$-ring} if the right (equivalently left) annihilator of any set of homogeneous elements is generated by a homogeneous projection. 
\end{definition}

We point out some properties which follow from these definitions. First, all concepts in Definition \ref{definition_graded_Rickart_and_Baer_star} are left-right symmetric. Second, the projections from the definitions above are necessarily unique. 
Third, if $A$ is a graded $*$-ring which is Rickart $*$, then it is a graded Rickart $*$-ring. Similarly, if $A$ is Baer $*$, then $A$ is a graded Baer $*$-ring. 
These two claims follow from the graded  analogues of \cite[Proposition 1.11]{Berberian_web} and \cite[Proposition 1.24]{Berberian_web}. These graded analogues can be shown directly following the proofs of  \cite[Proposition 1.11]{Berberian_web} and \cite[Proposition 1.24]{Berberian_web} and replacing ``idempotent'' with ``homogeneous idempotent'', ``projection'' with ``homogeneous projection'', ``ring'' with ``graded ring'' and doing similar adjustments.
Lastly, if   
$A$ is a graded Rickart $*$-ring, then $A_0$ is a Rickart $*$-ring and if $A$ is a graded Baer $*$-ring, then $A_0$ is a Baer $*$-ring. In remark (5) of \ref{seven_remarks} and \S \ref{section_questions}, we present examples of Baer $*$ but not graded Baer $*$ and Rickart $*$ but not graded Rickart $*$-rings.  

The involution of a Rickart $*$-ring is necessarily proper. We show that an analogous statement holds for graded rings. We say that a graded $*$-ring is \emph{graded proper}, if $xx^*=0$ implies $x=0$ for any homogeneous element $x$. If $A$ is a graded Rickart $*$-ring, $x\in A_\gamma$ and $xx^*=0,$ then $x^*\in\ann_r(x)=pA$ for some homogeneous projection $p.$ Thus $x^*=px^*.$ ``Starring'' this relation implies that $x=(px^*)^*=xp=0.$ 

We note that if $A$ is graded and proper, then $A$ is graded proper and if $A$ is graded proper, then the zero-component $A_0$ is proper. 

The definitions of positive definite and $n$-proper rings can be generalized to graded $*$-rings also by requiring that elements considered in the definitions are homogeneous. A graded $*$-ring $A$ is graded $n$-proper if and only if $\M_n(A)(\gamma_1,\dots,\gamma_n)$ is graded proper for any $n$ and any $\gamma_1,\dots,\gamma_n\in \Gamma$ as can be shown by a proof analogous to the corresponding non-graded statement (see \cite[Lemma 2.1]{Gonzalo_Ranga_Lia}).  Note that the \emph{$*$-transpose} $(a_{ij})^*=(a_{ji}^*)$, for $(a_{ij}) \in \M_n(A)(\gamma_1,\dots,\gamma_n)$, gives the structure of a graded $*$-ring to $\M_n(A)(\gamma_1,\dots,\gamma_n)$ if $A$ is a graded $*$-ring. 

We shall also use the following graded version of \cite[Proposition 1.13]{Berberian_web}.  
\begin{proposition}\label{graded_star_regular}
Let $A$ be a unital, $\Gamma$-graded $*$-ring. The following conditions are equivalent.

\begin{enumerate}
\item $A$ is graded regular and a graded Rickart $*$-ring.

\item $A$ is graded regular and graded proper.

\item For any homogeneous $x\in A,$ there is a homogeneous projection $p$ such that $xA=pA.$
\end{enumerate} 
\end{proposition}
We showed that (1) implies (2). The remaining implications follow from the proof of \cite[Proposition 1.13]{Berberian_web} when ``idempotent'' is replaced with ``homogeneous idempotent'', ``projection'' with ``homogeneous projection'', ``regular'' with ``graded regular'', ``proper'' with ``graded proper'' and doing similar adjustments.

If a unital, graded $*$-ring satisfies the equivalent conditions of Proposition \ref{graded_star_regular}, we say that it is \emph{graded $*$-regular}.

\subsection{Locally Rickart and Baer rings}\label{subsection_local}
As we pointed out, a Rickart ring is necessarily unital. The main object of our interest, a Leavitt path algebra, is not necessarily unital. Thus, we generalize the annihilator-related properties which we consider to non-unital rings.   
 
The properties of being right (left) Rickart and Baer are passed to corners (\cite[Propositions 2.2, 2.3]{Berberian_web}) since if $X\subseteq eAe$ for $e=e^2\in A,$ and $\ann_r(X)=fA$ for some idempotent $f$, then $\ann_r^{eAe}(X)=ef eAe.$ The same argument can be used to show that the properties of being graded right (left) Rickart and graded Baer are passed to corners generated by homogeneous idempotents. Analogously, being (graded) Rickart $*$ and (graded) Baer $*$ are passed to corners generated  by (homogeneous) projections. These facts can be used to generalize the annihilator-related properties from unital to non-unital rings as follows. 

\begin{definition}
Let $A$ be a ring, possibly non-unital. The ring $A$ is {\em locally right Rickart} if $eAe$ is right Rickart for any idempotent $e.$ We analogously define locally left Rickart, locally Rickart, locally Baer, locally Rickart $*$, and locally Baer $*$, as well as the graded versions of all these concepts.  
\end{definition}

The local concepts we introduced are particularly meaningful for locally unital rings. Recall that a ring $A$ is {\em locally unital} 
if for each finite set $F$ of elements of $A$, there is an idempotent $u$ such that $ux=xu=x$ for all $x\in F.$ The set of all such idempotents $u$ is said to be a set of local units. If $A$ is also a $*$-ring and the local units are projections, then we say that $A$ is {\em locally $*$-unital.} 

If $A$ is $\Gamma$-graded, $A$ is {\em graded locally unital} if for each finite set $F$ of (homogeneous) elements of $A$, there is a homogeneous idempotent $u$ such that $ux=xu=x$ for all $x\in F.$  One can check that the statements with and without the word ``homogeneous'' in parentheses in the previous sentence are equivalent. If $A$ is also a $*$-ring and the graded local units are projections, we say that $A$ is {\em graded locally $*$-unital.} 

If $A$ is a locally unital ring with the set of local units $U,$ then for $A$ to be locally right Rickart, locally left Rickart or locally Baer, it is sufficient that $uAu$ is right Rickart, left Rickart or Baer respectively for all $u\in U.$ This is because for any idempotent $e,$ there is a local unit $u\in U$ such that $e\in uAu$ and so $eAe=euAue$ is a corner of $uAu.$ Since any of the annihilator-related properties which we consider transfer to corners, $eAe$ keeps the property which $uAu$ has. Analogous statements can be made for graded rings, for $*$-rings, and for graded $*$-rings.

In addition, the local definitions coincide with the usual ones for unital rings. For example, if $A$ is a unital ring, then $A$ is locally right Rickart if and only if $A$ is right Rickart. One direction follows by considering the corner generated by the identity and the other since the property of being right Rickart passes to corners. Analogous statements can be made for graded rings, for $*$-rings and for graded $*$-rings. 

We shall use the graded and involutive version of \cite[Theorem 7.55]{Lam} stating that a ring with no infinite set of nonzero, orthogonal idempotents is left (right) Rickart if and only if it is Baer. The proof of the graded version of this statement is analogous to the proof of  \cite[Theorem 7.55]{Lam} and we omit it. The involutive version, in which the idempotents are replaced by projections, Baer with Baer $*$ and Rickart with Rickart $*$, holds by an analogous proof. We formulate the graded and graded local versions of this statement and note that those can be shown using analogous proofs. 

\begin{proposition}\label{Rickart_star=Baer_star}
Let $A$ be a $\Gamma$-graded $*$-ring. If $A$ has no infinite set of nonzero, orthogonal, homogeneous projections, then $A$ is a graded Rickart $*$-ring if and only if $A$ is a graded Baer $*$-ring.

If no corner of $A$ generated by a homogeneous projection has an infinite set of nonzero, orthogonal, homogeneous projections, then $A$ is a graded locally Rickart $*$-ring if and only if $A$ is a graded locally Baer $*$-ring.
\end{proposition}

In \S \ref{section_baer}, 
we use the following lemma, showing an implication of diagram (\ref{diagram}) for locally unital rings. 
 
\begin{lemma}\label{semihereditary}
If $A$ is locally unital and right semihereditary, then $A$ is locally right Rickart.  
\end{lemma}
\begin{proof}
Since $A$ is right semihereditary, $xA$ is projective for any $x\in A$ and so the short exact sequence  
$0\longrightarrow \ann_r(x)\longrightarrow A\longrightarrow xA\longrightarrow0$ 
splits. Let $\phi$ be a map $A\to \ann_r(x)$ such that $\phi(y)=y$ for all $y\in \ann_r(x).$ If $u$ is a local unit for $x,$ the map $\psi: uAu\to uAu$ defined by $\psi(y)=u\phi(y)$ is a homomorphism of right $uAu$-modules. The element $\psi(y)$ is in $\ann_r(x)$ for any $y\in uAu$ since $x\psi(y)=xu\phi(y)=x\phi(y)=0.$ So, $\psi(y)\in \ann_r(x)\cap uAu=\ann_r^{uAu}(x).$ If $y\in \ann_r^{uAu}(x),$ then $y\in \ann_r(x)$ and so $\phi(y)=y.$ Thus we have that $\psi(y)=u\phi(y)=uy=y$ and 
the short exact sequence 
\[0\longrightarrow \ann^{uAu}_r(x)\longrightarrow uAu\longrightarrow xuAu\longrightarrow0\]
splits. 
\end{proof} 

We adapt a part of Proposition \ref{graded_star_regular} to (graded) locally $*$-unital rings. We note the following lemma first.

\begin{lemma}\label{locally_rickart_star}
A locally (graded) Rickart $*$-ring with a set of (graded) local $*$-units is (graded) proper. 
\end{lemma}
\begin{proof}
If $A$ is a locally Rickart $*$-ring, $x\in A,$ and $u$ a local $*$-unit such that $x\in uAu,$ then  $x^*\in uAu$ since $u$ is a projection. If $xx^*=0$ then the same relation holds in $uAu.$ Then $x=0$ since $uAu$ is Rickart $*$ and, as such, proper.  The graded version of the lemma follows analogously. 
\end{proof}

\begin{proposition}\label{locally_star_regular}
Let $A$ be a (graded) $*$-ring with a set of (graded) local $*$-units. Then the following conditions are equivalent.
\begin{enumerate}
\item $A$ is (graded) regular and a (graded) locally Rickart $*$-ring.

\item $A$ is (graded) regular and (graded) proper.
\end{enumerate}  
\end{proposition}
\begin{proof}
The condition (1) implies (2) by Lemma \ref{locally_rickart_star}. For the converse, if $A$ is (graded) regular, then each corner of $A$ is (graded) regular as well. If $p$ is a (homogeneous) projection, then $pAp$ is a (graded) $*$-ring and, since $A$ is (graded) proper, $pAp$ is also (graded) proper. Thus, $pAp$ is (graded) regular and (graded) proper. In the non-graded case, $pAp$ is a Rickart $*$-ring by the analogue of Proposition  \ref{graded_star_regular} for non-graded rings (\cite[Proposition 1.13]{Berberian_web}). In the graded case, $pAp$ is a graded Rickart $*$-ring by Proposition \ref{graded_star_regular}. 
\end{proof}  

We point out that condition (2) is the same as in the unital case (compare with Proposition \ref{graded_star_regular}). This is not surprising since a locally unital ring is ``locally regular'' if and only if it is regular. Indeed, if $A$ is regular, then each corner of $A$ is regular. Conversely, if each corner of $A$ is regular, $x\in A,$ and $u=u^2$ is such that $x\in uAu,$ then there is $y\in uAu$ such that $xyx=x$ holds in $uAu$ and, hence, in $A$ also.

\section{Positive definite Leavitt path algebras}\label{section_positive_definite}
 
After a brief review of some graph-theoretic properties and the definition of a Leavitt path algebra, we focus on certain properties of the involution of a Leavitt path algebra. Specifically, we recall \cite[Proposition 2.4]{Gonzalo_Ranga_Lia} which asserts that a Leavitt path algebra is positive definite if and only if the underlying field is positive definite, the statement we use in Proposition~\ref{Rickart_characterization} and Theorems \ref{Baer_characterization} and \ref{Baer_star_characterization}. We show that an analogous statement holds for corner skew Laurent polynomial rings and, as a consequence, we provide an alternative proof of \cite[Proposition 2.4]{Gonzalo_Ranga_Lia}.  

Let $E=(E^0, E^1, \so, \ra)$  be a directed graph where $E^0$ is the set of vertices, $E^1$ the set of edges, and $\so, \ra: E^1
\to E^0$ are the source and the range maps. 

A vertex $v$ of a graph $E$ is said to be {\em regular} if the set $\so^{-1}(v)$ is nonempty and finite. A vertex $v$ is a {\em sink} if $\so^{-1}(v)$ is empty and a {\em source} if $\ra^{-1}(v)$ is empty. A graph $E$ is \emph{row-finite} if sinks are the only vertices which are not regular,
\emph{finite} if $E$ is row-finite and $E^0$ is finite (in which case $E^1$ is necessarily finite as well), and {\em countable} if both $E^0$ and $E^1$ are countable. 

A {\em path} $a$ in a graph $E$ is a finite sequence of edges $a=e_1\ldots e_n$ such that $\ra(e_i)=\so(e_{i+1})$ for $i=1,\dots,n-1$. Such
path $a$ has length $|a|=n.$  The maps $\so$ and $\ra$ extend to paths by $\so(a)=\so(e_1)$ and $\ra(a)=\ra(e_n)$  and $\so(a)$ and $\ra(a)$ are the \emph{source} and the \emph{range} of $a$ respectively. We consider a vertex $v$ to be a \emph{trivial} path of length zero with $\so(v)=\ra(v)=v$. A path $a = e_1\ldots e_n$ is said to be \emph{closed} if $\so(a)=\ra(a)$. If $a=e_1\ldots e_n$  is a closed path and $\so(e_i) \neq \so(e_j)$ for all $i \neq j$, then $a$ is a \emph{cycle}. A cycle of length one is a {\em loop}. A graph $E$ is said to be {\em no-exit} if $\so^{-1}(v)$ has just one element for every vertex $v$ every cycle visits. 

An infinite path of a graph $E$ is a sequence of edges $e_1e_2\ldots$ such that $\ra(e_i)=\so(e_{i+1})$ for $i=1,2,\ldots$. An infinite path is an \emph{infinite sink} if it has no cycles or exits. An infinite path  \emph{ends in a sink or a cycle} if there is a positive integer $n$ such that the subpath $e_ne_{n+1}\hdots$ is either an infinite sink or is equal to the path $cc\hdots$ for some cycle $c$ of positive length.  

Extend a graph $E$ to the graph with the same vertices and with edges $E^1\cup \{e^*\ |\ e\in E^1\}$ where the range and source relations are the same as in $E$ for $e\in E^1$ and $\so(e^*)=\ra(e)$  and $\ra(e^*)=\so(e)$ for the added edges. Extend the map $^*$ to all the paths by defining $v^*=v$ for all vertices $v$ and $(e_1\ldots e_n)^*=e_n^*\ldots e_1^*$ for all paths $a=e_1\ldots e_n.$  Extend also the maps $\so$ and $\ra$ to $a^*$ by $\so(a^*)=\ra(a)$ and $\ra(a^*)=\so(a)$.  

For a graph $E$ and a field $K$, the \emph{Leavitt path algebra} $L_K(E)$ of $E$ over $K$ is the free $K$-algebra generated by the set  $E^0\cup E^1\cup\{e^*\ |\ e\in E^1\}$ such that for all vertices $v,w$ and edges $e,f,$
\begin{itemize}
\item[(V)]  $vw =0$ if $v\neq w$ and $vv=v,$

\item[(E1)]  $\so(e)e=e\ra(e)=e,$

\item[(E2)] $\ra(e)e^*=e^*\so(e)=e^*,$

\item[(CK1)] $e^*f=0$ if $e\neq f$ and $e^*e=\ra(e),$

\item[(CK2)] $v=\sum_{e\in \so^{-1}(v)} ee^*$ for each regular vertex $v.$
\end{itemize}

The first four axioms imply that $L_K(E)$ is a $K$-linear span of the elements $ab^*$ where $a$ and $b$ are paths, and that $L_K(E)$ is unital if and only if $E^0$ is finite (with the sum of elements of $E^0$ as the identity). If $E^0$ is not finite, the finite sums of vertices are the local units of $L_K(E).$   

If $K$ is a field with involution $*$ (and there is always at least one such involution, the identity), the Leavitt path algebra $L_K(E)$ becomes an involutive algebra by $(\sum_{i=1}^n k_ia_ib_i^*)^* =\sum_{i=1}^n k_i^*b_ia_i^*$ for $k_i\in K$ and paths $a_i, b_i, i=1,\ldots, n.$
The algebra $L_K(E)$ is also naturally graded by $\Z$ with the $n$-component $L_K(E)_n$ equal to the $K$-linear span of elements $ab^*$ where $a, b$  are paths with $|a|-|b|=n$. This grading is such that $L_K(E)_n^*=L_K(E)_{-n}$ for every integer $n$ so $L_K(E)$ is a graded $*$-algebra. A Leavitt path algebra can be graded by an arbitrary abelian group also (see \cite{Hazrat_graded} or \cite{Roozbeh_Lia}) and all of our results can easily be adapted to any other grading of a Leavitt path algebra. 

A Leavitt path algebra of a finite graph with no sources can be represented as a corner skew Laurent polynomial ring, a special case of a fractional skew monoid ring considered in~\cite{arafrac}. Let $R$ be a unital ring, $p$ an idempotent of $R$ and $\phi:R\rightarrow pRp$ a unital ring isomorphism. Let $p_0 =1$ and $p_n =\phi^n(p_0)$ for $n=1,2,\ldots$. A \emph{corner skew Laurent polynomial ring}  $R[t_{+},t_{-},\phi]$ is a unital ring whose elements are formal expressions
\[t^m_{-}r_{-m} +t^{m-1}_{-}r_{-m+1}+\dots+t_{-}r_{-1}+r_0 +r_1t_{+}+\dots +r_nt^n_{+},\]
where $r_{-i} \in p_i R$ for $i=1, \ldots, m,$ and $r_j \in R p_j$ for $j=0,\ldots, n$. The addition is component-wise and the multiplication is determined by the distributive law and the following rules. 
\[t_{-}t_{+} =1 \qquad\;\; t_{+}t_{-} =p \qquad\;\; rt_{-} =t_{-}\phi(r)\qquad\;\;  t_{+}r=\phi(r)t_{+}\]

Assigning $1$ to $t_{+}$ and $-1$ to $t_{-}$ gives $A=R[t_{+},t_{-},\phi]$ a $\Z$-graded ring structure with $A=\bigoplus_{n\in \Z}A_n$ for 
\[A_0 =R,\;\; A_n= Rp_nt^n_{+},\;\; A_{-n}=t^{n}_{-}p_nR,\;\; n=1,2,\ldots.\] 
If $p=1$ and $\phi$ is the identity map, then $R[t_{+},t_{-},\phi]$ reduces to the ring of Laurent polynomials $R[t,t^{-1}]$. 

If $R$ is also an involutive ring and $\phi$ is a $*$-isomorphism (which implies that $p$ is a projection), then $A$ becomes a $*$-ring by $(t^i_{-}r_{-i})^*=r_{-i}^*t^i_{+}$ and $(r_it^i_{+})^*=t^i_{-}r_i^*$ for $i=1,2,\ldots,$ and $*$ extended to $A$ so that it is additive. In this way, $A$ becomes a graded $*$-ring. The following lemma relates some involutive properties of $R[t_{+},t_{-},\phi]$ to those of the coefficient ring $R$.

\begin{lemma}\label{corner_skew_Laurent} 
Let $R$ be a unital $*$-ring, $p\in R$ a projection and $\phi:R\rightarrow pRp$ a unital $*$-ring isomorphism. 
\begin{enumerate}[\upshape(1)]
\item $R$ is proper if and only if $R[t_{+},t_{-},\phi]$ is graded proper. 

\item $R$ is positive definite if and only if $R[t_{+},t_{-},\phi]$ is positive definite.
\end{enumerate}
\end{lemma}
\begin{proof}
(1) If $A=R[t_{+},t_{-},\phi]$ is graded proper, then $A_0=R$ is proper. For the converse, suppose $x\in A^h$. Then $x=rt^n_{+}$ or $x=t^{n}_{-}r$ for some $r\in Rp_n$ or $r\in p_nR,$ or $x\in A_0=R$. In the first case, we have \[xx^*=rt^n_{+} t^{n}_{-}r^*=r\phi^n(p_0)r^*=rp_nr^*=rp_np_n^*r^*=(rp_n)(rp_n)^*.\] If $R$ is proper, $xx^*=0$ implies that $rp_n=0$. Then $x=rt^n_{+}=rp_nt^n_{+}=0$. The proof of the case $x=t^{n}_{-}r$ is analogous. The case $x\in R$ follows directly from the assumption. 

(2) If $A=R[t_{+},t_{-},\phi]$ is positive definite, then $R=A_0$ is also positive definite. For the converse, let $x_k \in A,$ $k=1,\ldots, l$ and  
$x_k=t^{m_k}_{-}r_{-{m_k}, k} +t^{{m_k}-1}_{-}r_{-{m_k}+1, k}+\dots+t_{-}r_{-1, k}+r_{0, k} +r_{1, k}t_{+}+\dots +r_{n_k, k}t^{n_k}_{+}$ for some 
$r_{-i,k}\in p_iR,$ $i=1,\ldots, m_k$ and $r_{j,k}\in Rp_j,$ for $j=0,\dots, n_k, k=1,\ldots, l.$ Then 
\[\phi^{-m_k}\left(r_{-m_k, k} {r}_{-m_k, k}^*\right)+\dots +\phi^{-1}\left(r_{-1, k} r_{-1, k}^*\right)+r_{0, k} r_{0, k}^* +r_{1, k} r_{1, k}^*+\dots+r_{n_k, k}r_{n_k, k}^*\]
is the constant term of $x_kx^*_k$ for $k=1,\ldots, l.$  
If $\sum_{k=1}^l x_kx^*_k=0,$ then the sum of these constant terms is zero also for $k=1,\ldots, l.$ Since $\phi$ is a $*$-isomorphism and $R$ is positive definite,  $r_{-i, k}=0$ for $i=1,\ldots, m_k,$ and $r_{j, k}=0$ for $j=0,\ldots, n_k,$ for all $k=1,\ldots,l.$ Thus $x_k=0$.  
\end{proof}

A Leavitt path algebra of a finite graph with no sources can be represented as a corner skew Laurent polynomial ring (by \cite[Lemma~2.4]{arafrac}, see also \cite{Ara_Cortinas}). Let $E$ be a finite graph with no sources and let $E^0=\{v_1,\dots,v_n\}.$ For each $i=1,\ldots,n$, we choose an edge $e_i$ such that $\ra(e_i)=v_i$ and consider $t_{+}=e_1+\ldots+e_n \in L_K(E)_1$. Then $t_{-}=t_{+}^*$ is the left inverse of $t_{+}$ and $t_{+}t_{-}$ is a homogeneous projection. By \cite[Lemma~2.4]{arafrac}, $L_K(E)=L_K(E)_0[t_{+},t_{-},\phi]$ where 
$\phi:L_K(E)_0\to t_{+}t_{-}L_K(E)_0 t_{+}t_{-},$ is given by $x\mapsto t_{+}xt_{-}.$

In \cite[Proposition 2.3]{Gonzalo_Ranga_Lia}, it is shown that the Leavitt path algebra $L_K(E)$ over a positive definite field $K$ is proper and, as a consequence, that   $L_K(E)$ is positive definite if and only if $K$ is positive definite (\cite[Proposition 2.4]{Gonzalo_Ranga_Lia}).
The proof of  \cite[Proposition 2.3]{Gonzalo_Ranga_Lia} generalizes Raeburn's result \cite[Lemma 1.3.1]{raeburn1}, stating the same claim but just for row-finite graphs without sinks. To drop these assumptions for countable graphs, the proof of \cite[Proposition 2.3]{Gonzalo_Ranga_Lia} uses the Desingularization Process and, to drop the assumption that the graph is countable, the proof uses the Subalgebra Construction. In Proposition \ref{positive_definite}, we prove the analogous result for finite graphs using Lemma~\ref{corner_skew_Laurent} and a construction of ``attaching hooks'' and extend it to row-finite graphs. The Desingularization Process and the Subalgebra Construction imply the result for arbitrary graphs as well. 

\begin{proposition}\label{positive_definite}
Let $E$ be a graph and $K$ a field. The Leavitt path algebra $L_K(E)$ is positive definite if and only if the field $K$ is positive definite.
\end{proposition}
\begin{proof}
If $L_K(E)$ is positive definite and $\sum_{i=1}^n k_ik_i^*=0$ for $k_i\in K, i=1,\ldots, n$, then $\sum_{i=1}^n k_ivk_i^*v^*=0$ for any vertex $v$ and so $k_iv=0.$ This implies $k_i=0$ since $v$ is in the basis of $L_K(E).$

Conversely, if $K$ is positive definite, then any matrix algebra over $K$ is positive definite and hence any ultramatricial algebra over $K$ is also positive definite. If $E$ is a finite graph with no sources, $L_K(E)_0$ is an ultramatricial algebra (\cite[the proof of Theorem 5.3]{Ara_Moreno_Pardo}), and so $L_K(E)_0$ is positive definite. By Lemma~\ref{corner_skew_Laurent}, $L_K(E)$ is positive definite. 

Consider a finite graph $E$ which may have sources now. Replace each source $u$ of $E$ by the vertex $v$ of the following graph $\qquad \xymatrix{\bullet^w\ar@(lu,ld)\ar[r] &\bullet^v}$ which we refer to as a ``hook'', and add the rest of the hook. The graph $E'$ obtained in this way has no sources and so $L_K(E')$ is positive definite. The inclusion $E \rightarrow E'$ induces a $*$-monomorphism $\phi:L_K(E)\rightarrow L_K(E')$. The image $\phi(L_K(E))$ is positive definite since $L_K(E')$ is positive definite. Thus $L_K(E),$ $*$-isomorphic to $\phi(L_K(E))$, is positive definite as well. 

The Leavitt path algebra of a row-finite graph is a direct limit of Leavitt path algebras of certain finite subgraphs (those which are complete in the sense of  \cite{Ara_Moreno_Pardo}) with the injective connecting maps. Thus, such algebra is also positive definite. The claim then follows for a Leavitt path algebra of an arbitrary graph using the Desingularization Process and the Subalgebra Construction just as in \cite[Proposition 2.4]{Gonzalo_Ranga_Lia}.   
\end{proof}

\section{Rickart, Baer and Baer *-Leavitt path algebras}\label{section_baer}

In this section, we present Leavitt path algebra characterizations of annihilator-related properties, their graded versions, as well as their local generalizations. We start with the characterization of Leavitt path algebras which are (locally) Rickart, graded Rickart and graded Rickart $*$-rings. 

\begin{proposition}\label{Rickart_characterization}
Let $E$ be a graph and $K$ a field. The algebra $L_K(E)$ is locally Rickart (thus also graded locally Rickart). If $K$ is positive definite, $L_K(E)$ is a graded
locally Rickart $*$-ring. 

As a corollary, the following conditions are equivalent. 
\begin{enumerate}
\item $E^0$ is finite.

\item $L_K(E)$ is a right (left) Rickart ring.

\item $L_K(E)$ is a graded right (left) Rickart ring.

\end{enumerate}
\noindent If $K$ is positive definite, these conditions are also equivalent with the following.
\begin{enumerate}
\item[(4)] $L_K(E)$ is a graded Rickart $*$-ring.
\end{enumerate}
\end{proposition}

\begin{proof} 
By \cite[Theorem 3.7]{Ara_Goodearl}, $L_K(E)$ is a hereditary ring and, thus, semihereditary also. By Lemma \ref{semihereditary}, $L_K(E)$ is locally Rickart. Since every graded ring which is locally Rickart is graded locally Rickart, $L_K(E)$ is graded locally Rickard. If $K$ is positive definite, then $L_K(E)$ is positive definite also by Proposition \ref{positive_definite}. The algebra $L_K(E)$ is graded regular by \cite[Theorem 9]{Hazrat_regular} and so $L_K(E)$ is a graded locally Rickart $*$-ring by Proposition \ref{locally_star_regular}.

Since $L_K(E)$ is a $*$-ring, the conditions (2) and (3) are left-right symmetric. Recall also that $L_K(E)$ is unital if and only if (1) holds. A unital, locally right Rickart ring is right Rickart so (1) implies (2). Since (2) implies (3) by Lemma \ref{graded_Rickart_lemma} and (3) implies that $L_K(E)$ is unital, which is equivalent to (1), the conditions (1), (2) and (3) are equivalent. A unital, graded locally Rickart $*$-ring is a graded Rickart $*$-ring, so (1) implies (4), and (4) implies that $L_K(E)$ is unital, which is equivalent to (1). 
\end{proof}

If $R$ is a unital ring and $\kappa$ a cardinal, let $\M_\kappa(R)$ denote the ring of infinite matrices over $A$, having rows and columns indexed by $\kappa$, with only finitely many nonzero entries. By \cite[Theorem 3.7]{AAPM}, if $E$ is a countable, row-finite graph and $K$ a field, the following conditions are equivalent. 
\begin{enumerate} 
\item $E$ is a no-exit graph such that every infinite path ends in a sink or a cycle.
\item There are countable sets $I$ and $J$ and countable cardinals $\kappa_i, i\in I,$ and $\mu_j, j\in J,$ such that 
\begin{equation}\label{nongraded_matrices}
L_K(E) \cong \bigoplus_{i\in I} \M_{\kappa_i} (K) \oplus \bigoplus_{j \in J} \M_{\mu_j} (K[x,x^{-1}]).\tag{i}
\end{equation}
\end{enumerate}
The isomorphism in formula (\ref{nongraded_matrices}) can be taken to be either ring or algebra isomorphism.  
In \cite[Corollary 32]{Zak_Lia}, it is shown that the assumption on countability of the graph $E$ can be dropped (in which case $I,$ $J,$ $\kappa_i$ and $\mu_j$ may be of arbitrary cardinalities) and that the isomorphism in formula (\ref{nongraded_matrices}) can be taken to be a $*$-isomorphism. The set $I$ corresponds to the number of sinks of $E,$ the set $J$ to the number of cycles of $E$, the cardinal $\mu_j$ in the summand $\M_{\mu_j}(K[x,x^{-1}])$ corresponds to the number of paths ending in a fixed (but arbitrary) vertex of the cycle indexed by $j\in J$ which do not contain that cycle and the cardinal $\kappa_i$ in the summand $\M_{\kappa_i} (K)$ to the number of paths ending in the sink indexed by $i\in I$ (with a bit more subtleties if the sink is infinite as explained in \cite{Zak_Lia}, but we will not need those subtleties here).  

By \cite{Roozbeh_Lia}, if $E$ is a row-finite graph such that condition (1) holds, then $L_K(E)$ is graded $*$-isomorphic to the algebra 
\begin{equation}\label{graded_matrices}
\bigoplus_{i\in I} \M_{\kappa_i} (K)(\ol\gamma_i) \oplus \bigoplus_{j \in J} \M_{\mu_j} (K[x^{n_j},x^{-n_j}])(\ol\delta_j)\tag{ii}
\end{equation}
where $I$ and $J$ can be of arbitrary cardinality, $\kappa_i, i\in I,$ and $\mu_j, j\in J,$ are cardinals, $n_j$ positive integers, $\ol\gamma_i\in \Z^{\kappa_i}$ for $i\in I,$ and $\ol\delta_j\in \Z^{\mu_j}$ for $j\in J.$ 

We consider annihilator-related properties of the matrix algebras which appear in formulas (\ref{nongraded_matrices}) and (\ref{graded_matrices}). 

\begin{lemma}\label{matrices}
Let $\kappa$ be an arbitrary cardinal and $K$ a $*$-field. 
\begin{enumerate}[\upshape(1)]
\item The algebras $\M_{\kappa} (K)$ and $\M_{\kappa} (K[x,x^{-1}])$ are locally Baer. 

\item The trivial grading of $K$ by $\Z$ and the $\Z$-grading of $K[x,x^{-1}]$ given by $K[x,x^{-1}]_m=\{kx^m \mid k\in K\}$  $m\in\Z,$ make the algebras $\M_{\kappa} (K)(\ol\gamma)$ and $\M_{\kappa} (K[x^{n},x^{-n}])(\ol\gamma)$ into graded $*$-rings for $n$ a positive integer, and $\ol\gamma\in \Z^{\kappa}.$ These algebras are graded regular and, if $K$ is positive definite, graded locally Baer $*$-rings.  

\item If $K$ is positive definite, the algebras $\M_{\kappa} (K)$ and $K[x,x^{-1}]$ are locally Baer $*$-rings.   

\item The algebra $\M_{\kappa} (K[x,x^{-1}])$ is a locally Rickart $*$-ring if and only if $\kappa =1.$
\end{enumerate}
\end{lemma}
\begin{proof}
The algebras $\M_n(K)$ and $\M_n(K[x,x^{-1}])$ where $n$ is a positive integer, are Baer rings. There are several ways to see this. One, for example, is to note that these are hereditary rings by showing this fact either by definition or by representing these algebras as Leavitt path algebras of appropriate row-finite graphs and noting that all such algebras are hereditary by \cite[Theorem 3.5]{Ara_Moreno_Pardo}. As hereditary rings, these algebras are also Rickart. In addition, these algebras have no infinite sets of nonzero, orthogonal idempotents. So, they are Baer by \cite[Theorem 7.55]{Lam}. The algebras $\M_{\kappa} (K)$ and $\M_{\kappa} (K[x,x^{-1}])$ as in (1) are locally Baer since the finite sums of matrix units $e_{ii},$ $i\in \kappa,$ constitute sets of local units such that the corners are isomorphic to $\M_n(K)$ and to $\M_n(K[x,x^{-1}])$ respectively.   

The algebras $\M_{\kappa} (K)(\ol\gamma)$ and $\M_{\kappa} (K[x^{n},x^{-n}])(\ol\gamma)$ as in (2) are graded regular. One could check this directly or use the  representations of these algebras as Leavitt path algebras over certain acyclic or comet graphs (see \cite{Hazrat_graded}) and then use \cite[Theorem 9]{Hazrat_regular}. If $K$ is positive definite, then these algebras are positive definite as well. This can also be checked directly or using the Leavitt path algebra representations and Proposition \ref{positive_definite}. Thus, these algebras are graded locally Rickart $*$-rings by Proposition \ref{locally_star_regular}. Proposition \ref{Rickart_star=Baer_star} implies that these algebras are graded locally Baer $*$-rings. 

To show (3), note that $\M_{\kappa} (K)$ is regular for any cardinal $\kappa$. This can also be seen directly or by representing this algebra as a Leavitt path   algebra over an acyclic graph and using \cite[Theorem 1]{Abrams_Rangaswamy}. If $K$ is positive definite, $\M_{\kappa} (K)$ is regular and proper, so $\M_{\kappa} (K)$ is a locally Rickart $*$-ring by Proposition \ref{locally_star_regular}. By Proposition \ref{Rickart_star=Baer_star}, $\M_{\kappa} (K)$ is a locally Baer $*$-ring.
The algebra $K[x,x^{-1}]$ is Baer by (1) and a Rickart $*$-ring since it is an integral domain. Thus, it is a Baer $*$-ring. 

One direction of (4) is obvious: if $\kappa=1,$ $K[x,x^{-1}]$ is a Rickart $*$-ring. To show the converse, note that if $\kappa$ is infinite, the algebra $\M_n(K[x,x^{-1}])$ where $n$ is a finite ordinal is a corner of $\M_\kappa(K[x,x^{-1}])$ generated by a projection. Thus, it is sufficient to show that if $\M_n(K[x,x^{-1}])$ is a Rickart $*$-ring then $n=1$. Assume that $A=\M_n(K[x,x^{-1}])$ is a Rickart $*$-ring. By ~\cite[Proposition~1.11]{Berberian_web}, this implies that $eA=ee^*A$ for each idempotent $e$ of $A.$ If $n>1,$ consider the idempotent $e=e_{11}+e_{12}(1+x)$. Then $ee^*=e_{11}(3+x+x^{-1})$ and $e_{11}=ee_{11}\in eA.$ Assuming that $e_{11}\in ee^*A$ implies that $e_{11}=e_{11}(3+x+x^{-1})a$ for some $a=(a_{ij})\in A$ where $a_{ij}\in K[x,x^{-1}]$ for $i,j=1,\ldots,n.$ Multiplying by $e_{11}$ on the right implies that  $e_{11}=(3+x+x^{-1})e_{11}ae_{11}.$ By equating the upper left corner of the matrices in this last equation, we obtain that $1=(3+x+x^{-1})a_{11}.$ This implies that $3+x+x^{-1}$ is invertible in $K[x,x^{-1}]$ which is a contradiction. Hence, $n=1.$  
\end{proof}

We characterize (locally) Baer, graded Baer, and graded Baer $*$ Leavitt path algebras now.  

\begin{theorem}\label{Baer_characterization}
Let $E$ be a graph and $K$ a field.  The following conditions are equivalent. 

\begin{enumerate}
\item $E$ is a row-finite, no-exit graph in which every infinite path ends in a sink or a cycle.

\item $L_K(E)$ is a locally Baer ring.

\item $L_K(E)$ is a graded locally Baer ring.
\end{enumerate}
\noindent If $K$ is positive definite, these conditions are also equivalent with the following.  

\begin{enumerate}
\item[(4)] $L_K(E)$ is a graded locally Baer $*$-ring. 
\end{enumerate}

\noindent
As a corollary, the following conditions are equivalent. 
\begin{enumerate}
\item[(1')] $E$ is a finite, no-exit graph.

\item[(2')] $L_K(E)$ is a Baer ring.

\item[(3')] $L_K(E)$ is a graded Baer ring.
\end{enumerate}
\noindent If $K$ is positive definite,  these conditions are also equivalent with the following. 

\begin{enumerate}
\item[(4')] $L_K(E)$ is a graded Baer $*$-ring. 
\end{enumerate}
\end{theorem}
\begin{proof}
First, we note that (1) implies (2) and, if $K$ is positive definite, (1) implies (4). Indeed, if $E$ is a graph as in (1), then $L_K(E)$ is $*$-isomorphic to an algebra as in formula (\ref{nongraded_matrices}). The algebra of the form $ \M_{\kappa_i} (K)$ or $\M_{\mu_j} (K[x,x^{-1}])$ is locally Baer by part (1) of Lemma \ref{matrices} and so $L_K(E)$ is locally Baer. If (1) holds, $L_K(E)$ is  graded $*$-isomorphic to an algebra as in formula (\ref{graded_matrices}). 
By part (2) of Lemma \ref{matrices}, if $K$ is positive definite, the algebra of the form $ \M_{\kappa_i} (K)(\ol\gamma_i)$ or $\M_{\mu_j} (K[x^{n_j},x^{-n_j}])(\ol\delta_j)$ is a graded locally Baer $*$-ring. Thus, $L_K(E)$ is a graded locally Baer $*$-ring.

Since (2) implies (3) and (4) implies (3) by definition, it remains to show that (3) implies (1).  

Let $A=L_K(E)$ be a graded locally Baer ring. We show that the following holds.  
\begin{enumerate}
\item[1.] $E$ has no cycles with exits.  
 
\item[2.] $E$ has no infinite emitters. 

\item[3.] All infinite paths of $E$ end in sinks or cycles. 
\end{enumerate}
Assuming the negation of any of the three conditions above, we shall produce a vertex $v$ such that the corner $vAv$ is not graded Baer. Recall that the corner $vAv$ is a $K$-linear span of all the elements of the form $ab^*$ where $a,b$ are paths of $E$ originating at $v$ and having the same range.

{\bf 1. $\mathbf{E}$ has no cycles with exits. } 
Suppose $E$ has a cycle $a=a_1a_2\dots a_n$ which has an exit edge $b$ so that $v=\so(b)=\so(a_1)=\ra(a_n)$. Consider the infinite set of nonzero, orthogonal, homogeneous idempotents $S=\{a^ibb^*{a^*}^i\mid i=0, 1,\ldots \}$ in $vAv$ where $a^0=v$ and the subsets 
$S_{\text{o}}=\{a^ibb^*{a^*}^i\mid i \text{ is odd} \}$ and $S_{\text{e}}=\{a^ibb^*{a^*}^i\mid i \text{ is even} \}$ of $S$. The ring $vAv$ is graded Baer and so there is a homogeneous idempotent $e\in vAv$ such that $$\ann^{vAv}_r(S_{\text{o}})=eAv.$$
Since  $S_{\text{e}} \subseteq \ann^{vAv}_r(S_{\text{o}})$,  $ a^ibb^*{a^*}^i\in eAv$, for $i$ even. Thus  
\[e \, a^ibb^*{a^*}^i=a^ibb^*{a^*}^i\mbox{ if } i \text{ is even \;\; and }\;\;a^ibb^*{a^*}^i \, e =0  \mbox{ if } i \text{ is odd.}\]

Write $e$ as $k_1x_1y_1^*+k_2x_2y_2^*+\dots+k_nx_ny_n^*$, where $x_j$ and $y_j$ are paths of the same length with $\ra(x_j)=\ra(y_j),$ $\so(x_j)=\so(y_j)=v,$ and $k_j\in K$ for $j=1, \ldots, n$. Also write $e$ as $e_1+e_2+e_3+e_4$ where 

\begin{tabular}{l}
$e_1$ is the sum of those $k_jx_jy_j^*$ with  $x_j=y_j=a^i$ for some $i=0,1,\ldots$, \\
$e_2$ is the sum of those $k_jx_jy_j^*$ where only $y_j$ is of the form $a^i$ for some $i=0,1,\ldots$, \\
$e_3$ is the sum of those $k_jx_jy_j^*$ where only $x_j$ is of the form $a^i$ for some $i=0, 1,\ldots$, \\
$e_4$ is the sum of the rest of the monomials.
\end{tabular}

We note that such representation may not be unique.
Note that the relation $(e_3+e_4)a^ib=0$ holds in $A$ for a sufficiently large $i$ by the form of the terms appearing in $e_3+e_4$ and the axiom (CK1). Thus,  the relation $(e_3+e_4)a^ibb^*{a^*}^i=0$ holds in $vAv.$ Similarly, $(e_2^*+e_4^*)a^ibb^*{a^*}^i=0$ holds in $vAv$ for a sufficiently large $i.$ Thus, 
\begin{align*}
& (e_1+e_2) a^ibb^*{a^*}^i  =a^ibb^*{a^*}^i\; \text{ for a sufficiently large even }i,\text{ and}\\
& (e_1^*+e_3^*) a^ibb^*{a^*}^i  =0\; \; \; \text{ for a sufficiently large odd }i.
\end{align*}

Choose an integer $m$ larger than the length of $x_j$ and $y_j$ for all $x_j$ appearing in $e_2$ and all $y_j$ appearing in $e_3.$ Then $a^m{a^*}^me_2=a^m{a^*}^me_3^*=0.$ 
This implies that  
\begin{align*}
& a^m{a^*}^m e_1   \, a^ibb^*{a^*}^i =a^ibb^*{a^*}^i \; \text{ for any sufficiently large, even } i>m,  \\
& a^m{a^*}^m e_1^* \, a^ibb^*{a^*}^i =0  \; \; \; \; \text{ for any sufficiently large, odd } i. 
\end{align*}

Represent $e_1$ as $\sum_{l=1}^t k_l a^{n_l} {a^*}^{n_l}$ for some positive integer $t$ and nonnegative integers $n_l,$ $l=1,\ldots, t.$ For large enough $m$ and odd $i>m,$ 
\[0=a^m{a^*}^m e_1^*\,  a^i bb^*{a^*}^i=\sum_{l=1}^t k_l^* a^m{a^*}^m \,  a^i bb^*{a^*}^i =(\sum_{l=1}^t k_l^*) a^i bb^*{a^*}^i\]
and so $(\sum_{l=1}^t k_l^*) a^i bb^*{a^*}^i a^i b=(\sum_{l=1}^t k_l^*) a^i b =0$ holds in $A.$ Since a set of paths is a linearly independent set, $\sum_{l=1}^t k_l^*=0$ and so $\sum_{l=1}^t k_l=0$ also. For large enough $m$ and even $i>m,$
\[a^ibb^*{a^*}^i=a^m{a^*}^m e_1\, a^i bb^*{a^*}^i=(\sum_{l=1}^t k_l) a^i bb^*{a^*}^i=0.\]
Thus, $0=a^ibb^*{a^*}^ia^ib=a^ib$ holds in $A$ for large enough even $i.$ This is a contradiction since a set of paths is a linearly independent set.

{\bf 2. $\mathbf{E}$ has no infinite emitters.} Suppose that $E$ has a vertex $v$ which is an infinite emitter and let $\{a_i\}_{i=1,2,\ldots}$ be a set of edges which $v$ emits.  Consider the infinite set of nonzero, orthogonal, homogeneous idempotents 
$S=\{a_ia_i^*\mid i=1,2\ldots \}$  in $vAv$ and the subsets $S_{\text{o}}=\{a_ia_i^*\mid i \text{ is odd} \}$ and $S_{\text{e}}=\{a_ia_i^*\mid i \text{ is even} \}$ of $S.$ Let $e$ be a homogeneous idempotent in $vAv$ such that 
$\ann^{vAv}_r(S_{\text{o}})=eAv.$ Then $a_ia_i^*e=0$ for $i$ odd. Since  $S_{\text{e}} \subseteq \ann^{vAv}_r(S_{\text{o}})$, $a_ia_i^*=ea_ia_i^*$ for $i$  even. 

Write $e$ as $k_1x_1y_1^*+k_2x_2y_2^*+\dots+k_nx_ny_n^*+kv$, where $x_j$ and $y_j$ are paths of the same positive length with $\ra(x_j)=\ra(y_j),$ $\so(x_j)=\so(y_j)=v,$ and $k_j, k\in K$ for $j=1,\ldots, n,$ and denote $k_1x_1y_1^*+k_2x_2y_2^*+\dots+k_nx_ny_n^*$ by $x$ so that $e=x+kv.$  Again, we do not claim that such representation is unique. 

Since there is a finite number of paths $x_j$ appearing in $x$, there is an edge $a_i$ which does not appear in $x_j$ for all $j=1,\ldots, n$ and we can choose it such that $i$ is odd. Then, $a_i^*x=0$ holds in $A$ and so $a_ia_i^*x=0$ holds in $vAv.$ Since $i$ is odd, $0=a_ia_i^*e=a_ia_i^*kv=ka_ia_i^*$. Multiplying by $a_i$ on the right, we obtain that $0=ka_i$ holds in $A$. This implies that $k=0$ since $a_i$ is in the basis of $A.$ Hence $e=x.$ 

Choose an edge $a_i$ which does not appear in $y_j$ for all $j=1,\ldots, n$ and we can choose it so that $i$ is even. Thus, $xa_i=0$ holds in $A$ and so $xa_ia_i^*=0$ holds in $vAv.$  Since $i$ is even, $a_ia_i^*=ea_ia_i^*=xa_ia_i^*=0.$ Multiplying by $a_i$ on the right, we have that $0=a_ia_i^*a_i=a_i$ holds in $A$. This is a contradiction since the edges are in the basis of $A.$ 

{\bf 3. All infinite paths of $\mathbf{E}$ end in sinks or cycles.} Let $a$ be an infinite path and assume that $a$ does not end in a sink or a cycle. We showed that $E$ is no-exit, so the only way that there is a cycle in $a$ is if $a$ ends in it. Since that is not the case, there must be an infinite number of vertices of $a$ with exits. Let us represent $a$ as  $a=a_1a_2\ldots$ for some paths $a_i$ of positive length such the range $v_i$ of each $a_i$ has an exit $b_i.$ Let $v=\so(a),$ and  $c_0=v,$ $c_1=a_1,$ $c_2=a_1a_2,$ $c_3=a_1a_2a_3, \ldots.$
Consider the infinite set of nonzero, orthogonal, homogeneous idempotents $S=\{ c_ib_ib_i^*c_i^* \mid i=1,2,\ldots\}$
and the subsets $S_{\text{o}}=\{c_ib_ib_i^*c_i^*\mid i\text{ is odd} \}$ and $S_{\text{e}}=\{c_ib_ib_i^*c_i^*\mid i\text{ is even} \}$ of $S.$ Let $e$ be a homogeneous idempotent of $vAv$ such that $\ann^{vAv}_r(S_{\text{o}})=eAv.$ 
Then $c_ib_ib_i^*c_i^*e=0$ for $i$ odd. Since $S_{\text{e}} \subseteq \ann^{vAv}_r(S_{\text{o}})$, $c_ib_ib_i^*c_i^*=e c_ib_ib_i^*c_i^*$ for $i$ even.

Write $e$ as $k_1x_1y_1^*+k_2x_2y_2^*+\dots+k_nx_ny_n^*$, where $x_j$ and $y_j$ are paths of the same length with $\ra(x_j)=\ra(y_j),$ $\so(x_j)=\so(y_j)=v,$ and $k_j\in K$ for $j=1,\ldots, n$.  Also write $e$ as $e_1+e_2+e_3+e_4$ where 

\begin{tabular}{l}
$e_1$ is the sum of those $k_jx_jy_j^*$ with $x_j=y_j=c_i$ for some $i=0,1,\ldots$,  \\
$e_2$ is the sum of those $k_jx_jy_j^*$ where only $y_j$ is of the form $c_i$ for some $i=0, 1,\ldots$, \\
$e_3$ is the sum of those $k_jx_jy_j^*$ where only $x_j$ is of the form $c_i$ for some $i=0, 1,\ldots$, \\
$e_4$ is the sum of the rest of the monomials.
\end{tabular}

Note that the relation $(e_3+e_4)c_ib_i=0$ holds in $A$ for a sufficiently large $i$ by the form of the terms appearing in $e_3$ and $e_4$ and the axiom (CK1). Thus, the relation  $(e_3+e_4)c_ib_ib_i^*c_i^*=0$ holds in $vAv.$ Similarly, $(e_2^*+e_4^*)c_ib_ib_i^*c_i^*=0$ holds in $vAv$ for a sufficiently large $i.$ So, 
\begin{align*}
& (e_1+e_2) c_ib_ib_i^*c_i^* =c_ib_ib_i^*c_i^*\; \text{for a sufficiently large even } i \text{ and}\\ 
& (e_1^*+e_3^*) c_ib_ib_i^*c_i^*=0\; \; \; \text{for a sufficiently large odd }i. 
\end{align*}

Choose an integer $m$ such that the length of $c_m$ is larger than the length of $x_j$ and $y_j$ for all $x_j$ appearing in $e_2$ and all $y_j$ appearing in $e_3.$ Then $c_mc_m^*e_2=c_mc_m^*e_3^*=0$ and so 
\begin{align*}
& c_mc_m^* e_1\, c_ib_ib_i^*c_i^*  =c_ib_ib_i^*c_i^* \; \text{ for any sufficiently large, even } i, i>m\\ 
& c_mc_m^* e_1^*\, c_ib_ib_i^*c_i^*  =0 \; \; \; \text{ for any sufficiently large, odd } i.\\ 
\end{align*}

Represent $e_1$ as $\sum_{l=1}^t k_{l} c_{n_l}c_{n_l}^*$ for some positive integer $t$ and nonnegative integers $n_l,$ $l=1,\ldots, t.$ For large enough $m$ and $i>m$ odd, we have that
\[0=c_mc_m^* e_1^*\, c_ib_ib_i^*c_i^*=\sum_{l=1}^t k_l^* c_mc_m^*\, c_ib_ib_i^*c_i^*=(\sum_{l=1}^t k_l^*) c_ib_ib_i^*c_i^*.\]
Using an analogous argument as before, this implies that $\sum_{l=1}^t k_l^*=0$ and so $\sum_{l=1}^t k_l=0.$ For large enough $m$ and $i>m$ even, 
$$c_ib_ib_i^*c_i^*=c_mc_m^* e_1\, c_ib_ib_i^*c_i^* =(\sum_{l=1}^t k_l) c_ib_ib_i^*c_i^* = 0.$$
The relation $c_ib_ib_i^*c_i^*=0$ implies a contradiction as in the previous cases.

This finishes the proof of implication (3) $\Rightarrow$ (1) and the proof of the equivalence of the conditions (1) to (4). The equivalence of (1') to (4') follows from the equivalence of (1) to (4) as follows. If (1') holds, then (1) holds as well and $L_K(E)$ is unital. The condition (1) implies (2) and, if $K$ is positive definite, (4). If $L_K(E)$ is unital, (2) and (4) imply (2') and (4'). So, (2') and (4') follow from (1'). 

The implications (2') $\Rightarrow$  (3') and (4') $\Rightarrow$ (3') clearly hold. Finally, (3') implies (3) and that $L_K(E)$ is unital so $E^0$ is finite. Since (3) implies (1), (1) holds and $E^0$ is finite. A row-finite graph with finite $E^0$ is finite, thus condition (1') holds.   
\end{proof}

We characterize Leavitt path algebras which are (locally) Baer $*$-rings now. 

\begin{theorem}\label{Baer_star_characterization}
Let $E$ be a graph and $K$ be a positive definite field. The following conditions are equivalent.

\begin{enumerate}
\item $E$ is a disjoint union of graphs which are either acyclic with each infinite path ending in a sink 
or which are isolated loops. 

\item $L_K(E)$ is a locally Baer $*$-ring. 
\end{enumerate}
\noindent As a corollary, the following conditions are equivalent. 

\begin{enumerate}
\item[(1')] $E$ is a finite disjoint union of graphs which are finite and acyclic or isolated loops. 

\item[(2')] $L_K(E)$ is a Baer $*$-ring. 
\end{enumerate}
\end{theorem}
\begin{proof}
Assuming that (1) holds, $L_K(E)$ is $*$-isomorphic to an algebra as in formula (\ref{nongraded_matrices}) where $\mu_j=1$ for all $j\in J$. The matrix algebras in such representation are locally Baer $*$-rings by part (3) of Lemma \ref{matrices}. Thus, (2) holds. 

Conversely, assume that $L_K(E)$ is a locally Baer $*$-ring. Then  $L_K(E)$ is locally Baer and so $E$ is a row-finite, no-exit graph in which every infinite path ends in a sink or a cycle by Theorem~\ref{Baer_characterization}. The Leavitt path algebra $L_K(E)$ is $*$-isomorphic to an algebra as in formula (\ref{nongraded_matrices})
where $J$ is the number of cycles of $E$ and $\mu_j$ in the summand $\M_{\mu_j} (K[x,x^{-1}])$ corresponds to the number of paths ending in a fixed (but arbitrary) vertex of the cycle indexed by $j\in J.$ If $L_K(E)$ is a locally Baer $*$-ring, then each algebra $\M_{\mu_j}(K[x,x^{-1}])$ appearing in the representation (\ref{nongraded_matrices}) is such also. By part (4) of Lemma \ref{matrices}, this implies that $\mu_j=1$ for each $j$ and so the number of paths ending in any vertex of any cycle of $E$ is one. Hence, each cycle of $E$ is an isolated loop. 

The equivalence of (1') and (2') follows from the equivalence of (1) and (2) as follows. If (1') holds, then (1) holds as well and $L_K(E)$ is unital. The condition (1) implies (2) and (2) implies (2') if $L_K(E)$ is unital. So, (2') follows from (1').  Conversely, (2') implies (2) and that  $L_K(E)$ is unital so $E^0$ is finite. Since (2) implies (1), (1) holds and $E^0$ is finite. A row-finite graph with finite $E^0$ is necessarily finite, and so condition (1') holds.   
\end{proof}
To summarize, if $K$ is a positive definite field and $E$ is a connected graph, we have the following. 

{\em
\begin{center}
\begin{tabular}{|ccccc|}\hline
$L_K(E)$ is a Baer $*$-ring & $\Longrightarrow$ & $L_K(E)$ is Baer & $\Longrightarrow$ & $L_K(E)$ is Rickart\\
$\Updownarrow$&&$\Updownarrow$&&$\Updownarrow$\\
$E$ is finite and acyclic or a loop &$\Longrightarrow$& $E$ is finite and no-exit &$\Longrightarrow$& $E^0$ is finite\\ \hline
\end{tabular}
\end{center}}
Using the graph-theoretic properties in the bottom row of this diagram, it is very easy to create examples of algebras which illustrate that each implication in the first row is strict. For example, for any positive definite field, the first graph below produces an example of a Leavitt path algebra which is Baer but not a Baer $*$-ring and the second graph produces an example of a Leavitt path algebra which is Rickart but not Baer.   
$$ 
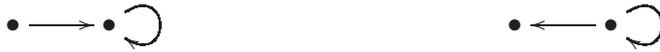
\xymatrix{ \bullet \ar[r] & \bullet\ar@(ru,rd) &\hskip3cm& \bullet  & \bullet\ar@(ru,rd)\ar[l] }$$ 

We reflect on the results from this section in a series of remarks.

\subsection{Remarks}\label{seven_remarks}
\begin{enumerate}[\upshape(1)]
\item The assumption that $K$ is positive definite is necessary in the relevant parts of Proposition~\ref{Rickart_characterization} and Theorem~\ref{Baer_characterization} as well as in Theorem \ref{Baer_star_characterization}. Let $K$ be any field which is not 2-proper (for example the field of complex numbers with the identity involution) nor graded 2-proper (any trivially graded and not 2-proper field is such). Then $\M_2(K)$ is not graded proper (see \S\ref{subsection_graded_star}) and thus not a graded Rickart $*$-ring. The algebra $\M_2(K)$ is graded $*$-isomorphic to the Leavitt path algebra over the graph $\xymatrix{\bullet\ar[r] &\bullet}.$
This graph satisfies conditions (1) of Proposition \ref{Rickart_characterization}, (1') of Theorem \ref{Baer_characterization} and (1') of Theorem \ref{Baer_star_characterization}. 

\item Theorem \ref{Baer_star_characterization} exhibits another difference between Leavitt path algebras and their graph $C^*$-algebra counterparts. 
An $AW^*$-algebra (a $C^*$-algebra which is a Baer $*$-ring) is separable if and only if it is finite dimensional. Since a graph $C^*$-algebra of a countable graph is separable (see~\cite[Remark~1.26]{raeburn}) and it is finite dimensional if and only if the associated graph is finite and acyclic (by~\cite[Proposition~3.4]{abramstomforde}),  
\begin{center}{\em 
a graph $C^*$-algebra $C^*(E)$ is a Baer $*$-ring if and only if $E$ is finite and acyclic. }
\end{center}
This contrasts Theorem~\ref{Baer_star_characterization} and, if $E$ is a loop $ \xymatrix{ \bullet\ar@(ru,rd)}\;\;\;\;\;$
and we consider $\C$ with the complex-conjugate involution, $L_\C(E)$ is a Baer $*$-ring while $C^*(E)$ is not. 

The referee of the paper noted a more direct argument for the fact that $L_\C(E)$ is a Baer $*$-ring while $C^*(E)$ is not: there are no nontrivial annihilators in the Laurent polynomial ring so it is trivially a Bear $*$-ring. On the other hand, the algebra of continuous functions on a sphere $C(S^1)$ has many nontrivial and proper annihilators, but only contains two projections, the constant functions 0 and 1.  

\item Every right (left) ideal generated by an idempotent in a $C^*$-algebra can be generated by a projection (\cite[Proposition IV.1.1]{Davidson}). Thus, for the graph $C^*$-algebras, being Rickart is equivalent with being a Rickart $*$-ring and being Baer is equivalent with being a Baer $*$-ring. The Leavitt path algebras do not have either of these two properties. Indeed, if $E$ is the following graph $\xymatrix{ \bullet \ar[r] & \bullet\ar@(ru,rd) }\;\;\;\;$
and we consider $\C$ with the complex-conjugate involution, then $L_\C(E)$ is Baer and not Baer $*.$ Thus, 
\begin{center}
\begin{tabular}{l}
{\em $C^*(E)$ is Baer if and only if $C^*(E)$ is a Baer $*$-ring}\hskip1.2cm while\\
{\em $L_K(E)$ can be Baer and not a Baer $*$-ring even if $K$ is positive definite.}
\end{tabular}
 
\end{center}
In \S\ref{section_questions}, we exhibit a graph $E$ such that $L_\C(E)$ is Rickart but not a Rickart $*$-ring. 

\item The proof of Theorem \ref{Baer_characterization} reveals another difference between the properties of Leavitt path algebras and the properties of $C^*$-algebras.
Recall that a Rickart $C^*$-algebra is a $C^*$-algebra which is a Rickart $*$-ring. 
As it was pointed out in the introduction to \cite{Kaplansky}, the annihilator of a countable set of elements of such algebra is generated by a projection. This does not have to hold for a Leavitt path algebra which is a Rickart $*$-ring. Indeed, if $E$ is a graph with two vertices $v$ and $w$ and countably many edges $e_n, n=1,2,\ldots,$ from $v$ to $w,$ 
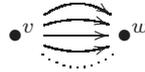
$$\xymatrix{{\bullet}^{v} \ar@{.} @/_1pc/ [r] _{\mbox{ } } \ar@/_/ [r] \ar [r] \ar@/^/ [r] \ar@/^1pc/ [r] & {\bullet}^{w}}$$ 
and $K$ is a positive definite field, then $L_K(E)$ is regular (by \cite[Theorem 1]{Abrams_Rangaswamy}) and proper and hence a Rickart $*$-ring. However, the set of elements $e_ne_n^*$ for $n$ odd, is a countable set  which is not generated by a homogeneous idempotent by Theorem \ref{Baer_characterization} nor by any idempotent as an appropriate modification of the proof of Theorem \ref{Baer_characterization} can show. This example also exhibits a Leavitt path algebra which is a Rickart $*$-ring but not Baer by Theorem \ref{Baer_characterization}.  

This example and the one from remark (1), exhibiting a Leavitt path algebra which is Baer but not a Rickart $*$-ring, show that the classes of Rickart $*$-rings and Baer rings are independent of each other for Leavitt path algebras.

\item By Theorem \ref{Baer_characterization}, a Leavitt path algebra is graded (locally) Baer if and only if it is (locally) Baer. However, using the graph-theoretic properties of Theorems \ref{Baer_characterization} and \ref{Baer_star_characterization}, it is easy to construct an example of a Leavitt path algebra which is a graded Baer $*$-ring and not a Baer $*$-ring.

This shows that the graded structure of Leavitt path algebras is less expected than it may seem at first: if (P) is the property of being Baer and (Q) the property of being a Baer $*$-ring then  {\em
\begin{center}
\begin{tabular}{l}
a Leavitt path algebra has (P)  if and only if it has graded (P)\\
a Leavitt path algebra can have graded (Q) but not have (Q).
\end{tabular}
\end{center}}
\noindent The last sentence is also true if (P) is the property of being Rickart (by Proposition \ref{Rickart_characterization}) and (Q) is the property of being regular (by \cite[Theorem 9]{Hazrat_regular} and \cite[Theorem 1]{Abrams_Rangaswamy}) or being a Rickart $*$-ring (as we show in \S \ref{section_questions}). 

\item An open conjecture of Handelman states that a $*$-regular ring is directly finite and unit-regular  (\cite[Question~48, p. 349]{Goodearl_book}). By results of  \cite{Gonzalo_Ranga_Lia} and \cite{Lia_Traces}, a counterexample to this conjecture cannot be found in the class of Leavitt path algebras. However, one can show that 
\begin{center}
{\em the graded version of Handelman's Conjecture fails}. 
\end{center}

Recall that a graded ring is \emph{graded unit-regular} if for any homogeneous element $x$, there is an invertible homogeneous element $y$ such that $xyx=x$. A graded ring is {\em graded directly finite} if for any homogeneous elements $x$ and $y,$ $xy=1$ implies $yx=1.$ Consider $E$ to be the graph  $\;\;\;\;\xymatrix{\ar@(lu,ld)\bullet\ar@(ru,rd) }\;\;\;\;\;$ with one vertex and two edges $a$ and $b.$  
If $K$ is a positive definite field, the algebra $L_K(E)$ is graded regular and proper by  \cite[Theorem 9]{Hazrat_regular} and Proposition~\ref{positive_definite}, and so it is graded $*$-regular by Proposition \ref{graded_star_regular}. However, there is no homogeneous invertible element $y$ such that $aya=a$ so $L_K(E)$ is not graded unit-regular. Also, $a^*a=1$ but $aa^*\neq 1$ so the algebra is not graded directly finite. 

\item Every Baer (and Baer $*$) ring can be decomposed into five factors, each of one of the five types described in \cite[\S 15]{Berberian} and \cite[\S 8]{Berberian_web}. Since rings of three of the five types are not directly finite and all Baer Leavitt path algebras are directly finite by \cite[Theorem 4.12]{Lia_Traces}, just two types, called $I_f$ and $II_1,$ are possible. However, the algebras of the form $\M_n(K)$ and $\M_n(K[x,x^{-1}])$ are of type $I_f$ since their identities are faithful finite projections (by definitions in \cite[\S 15]{Berberian} and \cite[\S 8]{Berberian_web}). Thus,
\begin{center}
{\em a Baer Leavitt path algebra is a Baer ring of type $I_f.$} 
\end{center}
\end{enumerate}

\section{Leavitt path algebras which are Rickart *-rings}\label{section_questions}

The previous section provides Leavitt path algebra characterizations of all annihilator-related properties we introduced {\em except} the property of being a (locally) Rickart $*$-ring. In this section we provide some input possibly leading to such characterization. 

Recall that a unital $*$-ring $A$ is said to be $*$-\emph{symmetric} if $1+xx^*$ is invertible in $A$ for each $x\in A.$ In a ring $A$ with this property, for each idempotent $e,$ there is a projection $p$ such that $eA = pA$ (\cite[Lemma 1.34]{Berberian_web}). Thus, a $*$-symmetric ring $A$ is Rickart if and only if $A$ is Rickart $*$ and $A$ is Baer if and only if $A$ is Baer $*$. We prove a lemma which helps us detect Leavitt path algebras which are not Rickart $*$-rings.  

\begin{lemma}\label{matrix_lemma}
If $A$ is a $*$-ring which is not $*$-symmetric, then $\M_n(A)$ is not a Rickart $*$-ring for any integer $n>1.$  
\end{lemma}
\begin{proof}
Let $n>1$ and let $x\in A$ be such that $1+xx^*$ is not invertible. If $p=e_{11}+xe_{12},$ then $p$ is an idempotent such that $e_{11}=pe_{11}\in pA$ and $e_{11}\notin pp^*A$ since $1+xx^*$ is not invertible. Thus, $A$ is not a Rickart $*$-ring since \cite[Proposition~1.11]{Berberian_web} implies that $eR=ee^*R$ for each idempotent $e$ of a Rickart $*$-ring $R.$
\end{proof}

If $K$ is any $*$-field, $K[x,x^{-1}]$, with the involution as in Lemma \ref{matrices}, is not $*$-symmetric since $1+(1+x)(1+x^{-1})=3+x+x^{-1}$ is not invertible. Note that this argument has been used in the proof of part (4) of Lemma \ref{matrices}. In fact, Lemma \ref{matrix_lemma} provides an alternative proof of part (4) of Lemma \ref{matrices}.  

We use Lemma \ref{matrix_lemma} to exhibit a Leavitt path algebra which is Rickart but not a Rickart $*$-ring. Let $T$ denote the Toeplitz algebra. We can represent this algebra as a Leavitt path algebra over the graph $E_T$ with two vertices $u$ and $v$ below.  
$$ 
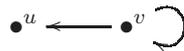
\xymatrix{ \bullet^u  & \bullet^v\ar@(ru,rd)\ar[l] }$$ 
The property of being $*$-symmetric is passed onto quotients. Since $K[x,x^{-1}]$, $*$-isomorphic to the Leavitt path algebra of the quotient graph $E_T/\{u\},$ is a quotient of $T$ and $K[x,x^{-1}]$ is not $*$-symmetric, $T$ is not $*$-symmetric also. Thus, $\M_n(T)$ is not a Rickart $*$-ring for any $n>1$ by Lemma \ref{matrix_lemma}. Note that $\M_n(T)$ can be represented as the Leavitt path algebra of the graph obtained by attaching a line of length $n-1$ to both vertices of the graph $E_T.$ For example, $\M_2(T)$ is $*$-isomorphic to the Leavitt path algebra of the graph below. 
$$ 
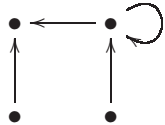
\xymatrix{ \bullet  & \bullet\ar@(ru,rd)\ar[l]\\
\bullet\ar[u]&\bullet\ar[u]}$$ 
This construction and Lemma \ref{matrix_lemma} exhibit Leavitt path algebras which are Rickart rings but not Rickart $*$-rings even over a positive definite field. 

We finish the paper with the following problem.
\begin{center}
{\em Find a graph property which characterizes Leavitt path algebras which are (locally) Rickart $*$-rings. } 
\end{center}
In other words, if $E$ is a connected graph and $K$ a positive definite field, find a property replacing the question mark in the diagram below. 
This diagram also summarizes the relationships between the classes of Leavitt path algebras (LPAs). As we demonstrated, all the implications in this diagram are strict. 
{\small
$$\xymatrix{&\text{
\begin{tabular}{c}Baer LPAs\\ $\Updownarrow$\\ $E$ is finite and no-exit 
\end{tabular}
}\ar@2[dr]&\\ \text{
\begin{tabular}{c}
Baer $*$ LPAs \\ $\Updownarrow$\\ $E$ is finite and acyclic or a loop 
\end{tabular} }\ar@2[ur]\ar@2[dr] && \text{
\begin{tabular}{c}
Rickart LPAs \\ $\Updownarrow$\\ $E^0$ is finite
\end{tabular}}\\&\text{
\begin{tabular}{c}
Rickart $*$ LPAs \\ $\Updownarrow$\\$E$ is ? 
\end{tabular}}\ar@2[ur]&}$$
}

\end{document}